\documentclass[12pt]{article}
\usepackage{amsfonts,amsmath,amsthm,amscd,amssymb,latexsym,hyperref}

\usepackage{xcolor}
\usepackage{tikz}
\usetikzlibrary{positioning,automata}
\usetikzlibrary{arrows,calc}
\usetikzlibrary{decorations.markings,plotmarks}

\begin{document}

\title{On a finite state representation of $GL(n,\mathbb{Z})$ }
\author{Andriy Oliynyk and Veronika Prokhorchuk}

\date{}


\theoremstyle{plain}
\newtheorem{theorem}{Theorem}
\newtheorem{lemma}{Lemma}
\newtheorem{proposition}{Proposition}
\newtheorem{corollary}{Corollary}
\newtheorem{definition}{Definition}
\theoremstyle{definition}
\newtheorem{example}{Example}
\newtheorem{remark}{Remark}

\maketitle

\begin{abstract}
It is examined finite state automorphisms of regular rooted trees constructed in~\cite{MR1492064} to represent groups $GL(n,\mathbb{Z})$.  The number of states of automorphisms that correspond to elementary matrices is computed. Using the representation of $GL(2,\mathbb{Z})$ over an alphabet of size $4$  a finite state  representation of the free group of rank $2$ over binary alphabet is constructed.
\end{abstract}

\section{Introduction}
\label{section_Introduction}

Representations of residually finite groups and semigroups by automorphisms and endomorphisms of regular rooted trees is an attractive and challenging research direction. It is inspired mainly by brilliant examples of infinite finitely generated periodic groups constructed as automorphism groups of rooted trees (\cite{MR0545692, MR0565099, MR0696534}). Results on ubiquity of free groups and semigroups in automorphism groups of rooted trees (\cite{MR1320624,MR1629906}) stimulated explicit representations of groups and semigroups containing free subgroups and subsemigroups. Among them, in~\cite{MR1492064} a natural and brilliant representation of groups $GL(n,\mathbb{Z})$, $n>1$, by finite state automorphisms of $2^n$-regular rooted tree was found.

The purpose of this note is further investigation of the construction presented in~\cite{MR1492064}. For elementary matrices in $GL(n,\mathbb{Z})$, $n>1$, it is calculated the number of states of finite state automorphisms corresponding to them. An algorithm to construct a finite state automorphism  by a given unimodular matrix is discussed and implemented. Finally, it is applied a method from~\cite{ADM2023} to construct a quite surprising representation of the free group of rank $2$ by finite state automorphisms of a regular rooted tree based on the representation of $GL(2,\mathbb{Z})$.

The organization of the paper is following. In Section~\ref{section_Rooted_trees} it is briefly recalled required notions and properties regarding finite state automorphisms of regular rooted trees. Details on rooted trees, their automorphisms and in particular finite state automorphisms one can find in~\cite{MR1841755,MR2162164}.
Section~\ref{section_Automaton_representation_GLnZ} contains computations of the number of states of automorphisms corresponding to elementary unimodular matrices.
In Section~\ref{section_Automaton_representation_free_group} a representation of the free group of rank $2$ is constructed.

\section{Rooted trees and their automorphisms}
\label{section_Rooted_trees}

Let $\mathcal{T}_n$, $n>1$, be a rooted $n$-regular tree. Denote by $\mathsf{X}$ the set of vertices of $\mathcal{T}_n$, connected with the root. Then $|\mathsf{X}|=n$. It is convenient to treat the tree $\mathcal{T}_n$ as the (right) Cayley graph of the free monoid $\mathsf{X}^*$ with basis $\mathsf{X}$. From this point of view each vertex of  $\mathcal{T}_n$ is a finite word over $\mathsf{X}$, the root is the empty word $\Lambda$. Two words $u,v$ are connected by an edge if and only if $u=vx$ or $v=ux$ for some $x \in \mathsf{X}$.  

The automorphism group $Aut \mathcal{T}_n$ of  $\mathcal{T}_n$ is a permutational wreath product 
\[
Sym(\mathsf{X}) \wr Aut \mathcal{T}_n
\]
of the symmetric group on $\mathsf{X}$ with  $Aut \mathcal{T}_n$ itself. Each automorphism
$g \in Aut \mathcal{T}_n$ can be uniquely defined by a permutation $\sigma_g \in Sym(\mathsf{X})$ and a multiset $g_x \in Aut \mathcal{T}_n, x\in \mathsf{X}$ that form the so-called wreath recursion
\[
g=(g_x, x \in \mathsf{X})\sigma_g.
\]
The right action of $g$ on vertices of $\mathcal{T}_n$ can be written recursively as follows
\[
(xw)^g=x^{\sigma_g}w^{g_x}, \quad x \in \mathsf{X}, w \in \mathsf{X}^*.
\]
The permutation $\sigma_g$ is called the rooted permutation of $g$. Automorphisms $g_x, x\in \mathsf{X}$ are called states of the first level of $g$. 
Using wreath recursions the product of automorphisms 
\[
g = (g_x, x \in \mathsf{X})\sigma_g, \quad h = (h_x, x \in \mathsf{X})\sigma_h
\]
can be expressed as
\[
gh= (g_x h_{x^{\sigma_g}}, x \in \mathsf{X})\sigma_g\sigma_h.
\]
The identity automorphism will be denoted by $e$.

For arbitrary vertex  $v \in \mathsf{X}^*$ the state of $g$ at $v$ is a uniquely defined automorphism $g_v$ such that
\[
(vw)^g=v^{g}w^{g_v}, \quad w \in \mathsf{X}^*.
\]
The set $Q(g)=\{g_v: v \in \mathsf{X}^*\}$ is called the set of states of $g$. 
Since $g_{\Lambda}=g$ the automorphism $g$ is its state as well. If $Q(g)$ is finite the automorphism $g$ is called finite state automorphisms. 
All finite state automorphisms of $\mathcal{T}_n$ form a countable subgroup $FAut \mathcal{T}_n$ in $Aut \mathcal{T}_n$. We say that a group $G$ has a finite state representation if it is isomorphic to a subgroup of $FAut \mathcal{T}_n$ for some $n$. The self-similar closure of an automorphism $g$ is a subgroup of $Aut \mathcal{T}_n$ generated by the set $Q(g)$. Multiplication rule for automorphisms imply $Q(g^{-1})= \{h^{-1}: h \in Q(g)\}$.

For arbitrary automorphisms $g,h$ and a vertex $v$ the state of their product $gh$ at $v$ is the product $g_vh_{v^g}$. In particular, the set $Q(gh)$ of states of the product $gh$ is a subset of the product $Q(g)Q(h)$ of multipliers' sets of states.

Each automorphism $g \in Aut \mathcal{T}_n$ can be defined by its Moore diagram, i.e. a directed graph with $Q(g)$ as the vertex set. The vertex $g$ is marked. For arbitrary state $h\in Q(g)$ the Moore diagram of $g$ has exactly $n$ labelled arrows starting from $h$. For each $x \in \mathsf{X}$ exactly one arrow starts in $h$ and terminates in $h_x$. It has a label of the form $x | x^{\sigma_h} $. 

\section{Finite state representation of $GL(n,\mathbb{Z})$}
\label{section_Automaton_representation_GLnZ}

Let $n>1$. In~\cite{MR1492064} the authors constructively prove that the  group $GL(n,\mathbb{Z})$ is isomorphic to a subgroup of $FAut \mathcal{T}_{2^n}$. Let us recall this embedding.  We will identify the vertex set of $\mathcal{T}_{2^n}$ with the set of all finite words over the vector space $\mathbb{Z}_2^n$ of dimension $n$ over the binary field $\mathbb{Z}_2$. Define the following permutations $\tau, \sigma, \pi_{i,j}$, $1\le i<j \le n$:
\[
(x_1,x_2,\ldots, x_n)^{\tau}=(x_1+x_2,x_2, \ldots, x_n),
\]
\[
(x_1,x_2,\ldots, x_n)^{\sigma}=(x_1+1,x_2, \ldots, x_n),
\]
\[
(x_1,\ldots, x_i,\ldots, x_j, \ldots x_n)^{\pi_{ij}}=(x_1,\ldots, x_j,\ldots, x_i, \ldots x_n),
\]
where $(x_1,x_2,\ldots, x_n) \in \mathbb{Z}_2^n$. Then all these permutations and the product $\tau \sigma$ are involutions. Consider automorphisms $t_1,t_2$, $s_{i,j}$,$1\le i<j \le n$,  of $Aut \mathcal{T}_{2^n}$, defined by the following wreath recursions:
\[
t_1=(t_1{}_{(x_1,x_2,\ldots, x_n)}, (x_1,x_2,\ldots, x_n) \in \mathbb{Z}_2^n) \tau,
\]
\[
t_2=(t_2{}_{(x_1,x_2,\ldots, x_n)}, (x_1,x_2,\ldots, x_n) \in \mathbb{Z}_2^n) \tau \sigma 
\]
\[
s_{ij}=(s_{ij}{}_{(x_1,x_2,\ldots, x_n)}, (x_1,x_2,\ldots, x_n) \in \mathbb{Z}_2^n)\pi_{ij},
\]
where
\[
t_1{}_{(x_1,x_2,\ldots, x_n)}=
\begin{cases}
    t_2, & \text{ if } x_1=x_2=1 \\
    t_1, & \text{ otherwise }
\end{cases}, 
\]
\[
t_2{}_{(x_1,x_2,\ldots, x_n)}=
\begin{cases}
    t_1, & \text{ if } x_1=1, x_2=0 \\
    t_2, & \text{ otherwise }
\end{cases}, 
\]
$s_{ij}{}_{(x_1,x_2,\ldots, x_n)}=s_{ij}$. Then $Q(t_1)=Q(t_2)=\{t_1,t_2\}$, $Q(s_{ij})=\{s_{ij}\}$. 

It is shown in~\cite{MR1492064}, that the subgroup of $FAut \mathcal{T}_{2^n}$ generated by the set $\{t_1, s_{ij}, 1\le i<j \le n \}$ is isomorphic to $GL(n,\mathbb{Z})$. More precisely, the required isomorphism is defined as follows. Denote by $T_{ij}(k)$ the elementary $n \times n$ matrix obtained from the identity matrix by adding the $i$th column multiplied by $k$  to the $j$th column, $1\le i,j \le n$, $i\ne j$, $k \in \mathbb{Z}$, $k\ne 0$, and by $E_{i}$ the elementary matrix obtained from the identity by multiplying its $i$th column by $-1$, $1\le i \le n$. Denote by $E_{ij}$ the permutation matrix  that correspond to the transposition $(ij)$, $1\le i<j \le n$. 
Then the mapping $\varphi_n$ that sends elementary matrix $T_{21}(1)$ to $t_1$ and  permutation matrix $E_{ij}$ to $s_{ij}$, $1\le i<j \le n$, defines the required isomorphic embedding.

This construction gives rise to the following natural algorithm of constructing a finite state automorphism $\varphi_n(A)$ corresponding to a given matrix $A \in GL(n,\mathbb{Z})$:
\begin{enumerate}
    \item 
    factorize $A$ as a product of elementary matrices $F_1 \ldots F_m$;    
    \item
    compute finite state automorphisms $\varphi(F_i)$, $1\le i \le m$;
    \item 
    compute $\varphi(A)$ as the product $\varphi(F_1) \ldots (F_m)$.
\end{enumerate}

We implemented this algorithm using \texttt{GAP} (\cite{GAP4}) and \texttt{AutomGrp}\cite{AutomGrp1.3.2} package.

In order to estimate the number of states of $\varphi_n(A)$ we examine the subgroup generated by automorphisms $t_1,t_2$.

\begin{lemma}
    \label{t1_and_t2_commute}
    Automorphisms $t_1$ and $t_2$ commute.
\end{lemma}

\begin{proof}
Denote by $g$ and $h$ the products $t_1t_2$ and $t_2t_1$  correspondingly. Then their rooted permutations are $\tau \tau \sigma$ and $ \tau \sigma \tau$. Each of them equals  $\sigma$.
For arbitrary $(x_1,x_2,\ldots, x_n) \in \mathbb{Z}_2^n$ states of the first level of $g$ and $h$ at $(x_1,x_2,\ldots, x_n)$  have the form
\[
g_{(x_1,x_2,\ldots, x_n)}=
\begin{cases}
    t_1^2,& \text{ if } x_1=1, x_2=0 \\
    t_2^2,& \text{ if } x_1=1, x_2=1 \\
    t_1t_2, & \text{ otherwise }
\end{cases},
\]
\[
h_{(x_1,x_2,\ldots, x_n)}=
\begin{cases}
    t_1^2,& \text{ if } x_1=1, x_2=0 \\
    t_2^2,& \text{ if } x_1=1, x_2=1 \\
    t_2t_1, & \text{ otherwise }
\end{cases}.
\]
Since rooted permutations of $g$ and $h$ are equal we obtain by induction the equality $g=h$. The proof is complete. 
\end{proof}

\begin{lemma}
    \label{states_of_powers_of_t1_and_t2}
    Let $n\ge 1$ and $Q=\{t_1^n,t_1^{n-1}t_2, \ldots, t_1t_2^{n-1}, t_2^n \}$. Then automorphisms from $Q$ are pairwise different and for arbitrary $g\in Q$ the set of states of $g$ is $Q$.  In particular, automorphisms $t_1^n$ and $t_2^n$ have exactly $n+1$ states.
\end{lemma}

\begin{proof}
Let $g= t_1^{2k_1+\varepsilon_1}t_2^{2k_2+\varepsilon_2}$ for non-negative integers $k_1,k_2$ and $\varepsilon_1,\varepsilon_1 \in \{0,1\}$. Using induction on $n$ and Lemma~\ref{t1_and_t2_commute} one can directly verify that
for arbitrary $(x_1,x_2,\ldots, x_n) \in \mathbb{Z}_2^n$ the states of the first level of $g$ at $(x_1,x_2,\ldots, x_n)$ has the form
\begin{equation}
    \label{states_of_powers}
g_{(x_1,x_2,\ldots, x_n)}=
\begin{cases}
    t_1^{2k_1+\varepsilon_1+k_2}t_2^{k_2+\varepsilon_2},& \text{ if } x_1=x_2=0 \\
    t_1^{2k_1+\varepsilon_1+k_2+\varepsilon_2}t_2^{k_2},& \text{ if } x_1=1, x_2=0 \\
    t_1^{k_1+\varepsilon_1}t_2^{k_1+2k_2+\varepsilon_2}, & \text{ if } x_1=0, x_2=1 \\
    t_1^{k_1}t_2^{k_1+\varepsilon_1+2k_2+\varepsilon_2},& \text{ if } x_1=x_2=1 \\
\end{cases}.
\end{equation}
Since $t_1 \ne t_2$ by induction on $k$ using~(\ref{states_of_powers}) one obtains inequality $t_1^k \ne t_2^k$, $k \ge 1$. Hence, automorphisms from $Q$ are pairwise different and $|Q|=n+1$. 

Since states of the first level of $g$ belong to $Q$ the inclusion
$Q(g)\subseteq Q$ holds. To prove the equality it is sufficient to show that the Moore diagram of $g$ is a strongly connected graph. 
Direct computations show that this statement holds for $n\le 6$. Then we assume that $n>6$. Let $n$ 

Using states at $(1,0,\ldots, 0)$ one constructs a directed path from $g$ to $t_1^{n-1}t_2$ and $t_1^n$.  Now it is sufficient to show that for arbitrary  $l$, $0\le l \le n$, the state $t_1^lt_2^{n-l}$ is accessible from $t_1^n$. It follows from~(\ref{states_of_powers})  that accessibility of $t_1^lt_2^{n-l}$ implies accessibility of $t_1^{n-l}t_2^{l}$. Let us call this property symmetricity.

Assume from the contrary that there $n$ is the least positive integer such that 
there exists not accessible states from $Q$. Let $k$ be an integer such that state $t_1^{k}t_1^{n-k}$ is not accessible. Then $0<k<n$. Since the statement about accessibility holds for $n-1$ the state $t_1^{k}t_1^{n-k-1}$ is accessible from $t_1^{n-1}$. As soon as each state of the product is a product of states of the multipliers it means that in the product $t_1^{n-1} \cdot t_1$ the state $t_1^{k}t_1^{n-k-1}$ is multiplied by $t_1$ only. Hence, the $t_1^{k+1}t_1^{n-k-1}$ is accessible. Let us call this property inconsistency.

Assume now that $l$ is the least number such that $t_1^{l}t_1^{n-l}$ is not accessible. Then symmetricity implies $l<n/2$. Since $t_1t_2^{n-1}$ is accessible at least one of $t_1^2t_2^{n-2}$ and $t_1^3t_2^{n-3}$ is accessible. In both cases $t_1^2t_2^{n-2}$ is accessible.  Hence $2\le l$. 
Applying~(\ref{states_of_powers}) for cases $x_2=0$ one obtains that states
\[
t_1^{2l+1}t_2^{n-2l-1}, \quad t_1^{2l+2}t_2^{n-2l-2} 
\]
are not accessible. It contradicts with inconsistency. 
The proof is complete.
\end{proof}

\begin{proposition}
    \label{self-similar_closure}
    The self-similar closure of each of automorphisms $t_1$ and $t_2$ is isomorphic to $\mathbb{Z} \times \mathbb{Z}$.    
\end{proposition}

\begin{proof}
    The statement immediately follows from Lemma~\ref{t1_and_t2_commute} and Lemma~\ref{states_of_powers_of_t1_and_t2}.
\end{proof}

\begin{theorem}
    \label{estimation_of_automaton_size}
    Let $A \in GL(n,\mathbb{Z})$. Denote by $m(A)$ the number of states of the automorphism $\varphi_n(A) \in FAut \mathcal{T}_{2^n}$.
    \begin{enumerate}
        \item
        \label{thm_item1}
        If $A$ is a permutation matrix then $m(A)=1$.
        \item 
        \label{thm_item2}
        If $A=E_i$, $1\le i \le n$, then $m(A)=8$.
        \item
        \label{thm_item3}
        If $A=T_{ij}(k)$, $1\le i,j \le n$, $i\ne j$, $k \in \mathbb{Z}$, $k\ne 0$, then $m(A)=|k|+1$.                
    \end{enumerate}
\end{theorem}

\begin{proof}
\ref{thm_item1}. 
Definition of $\varphi_n$ implies $m(E_{ij})=1$ for a permutation matrix $E_{ij}$, $1\le i<j \le n$. Since the product of automorphisms with exactly 1 state has 1 state the required equality holds for arbitrary permutation matrix.

\ref{thm_item2}.
Since automorphism $t_1$ has 2 states its inverse $t_1^{-1}=\varphi_n(T_{21}(-1))$ has 2 states as well. Then the equality
\[
E_1=T_{21}(1)\cdot E_{12} \cdot T_{21}(-1) \cdot E_{12} \cdot T_{21}(1) \cdot E_{12}.
\]
implies $m(E_1)\le 8$. Direct verification shows that equality holds.
Since $E_i=E_{1i}\cdot E_1\cdot E_{1i}$ one obtains  $m(E_i)= 8$, $2\le i \le n$.

\ref{thm_item3}. 
Since $\varphi_n(T_{21}(1))=t_1$ Lemma~\ref{states_of_powers_of_t1_and_t2} implies 
$m(T_{21}(k))=k+1$, $k>0$. The inverse automorphism  $t_1^{-k}$ equals $\varphi_n(T_{21}(-k))$ and has $k+1$ states as well. Hence $m(T_{21}(-k))=k+1$, $k>0$.

Then from equalies 
\[
T_{12}(k)=E_{12}\cdot T_{21}(k)\cdot E_{12}, 
\] 
\[
T_{i1}(k)=E_{2i}\cdot T_{21}(k)\cdot E_{2i}, \quad 3\le i \le n,
\]
\[
T_{2j}(k)=E_{1j}\cdot T_{21}(k)\cdot E_{1j}, \quad 3\le j \le n,
\]
\[
T_{ij}(1)=E_{2i}\cdot E_{1j}\cdot T_{21}(1) \cdot E_{1j}\cdot E_{2i}, \quad
3\le i,j \le n, i\ne j,
\]
it follows 
$m(T_{ij}(k))=|k|+1$, $1\le i,j \le n$, $i\ne j$, $k\in \mathbb{Z}$, $k\ne 0$.
\end{proof}

This theorem together with a factorization of a matrix $A \in GL(n,\mathbb{Z})$ in a product of elementary matrices give an upper estimation on the size of $Q(\varphi_n(A))$. In particular, from the third statement of the theorem we immediately have

\begin{corollary}
    Let $A \in SL(n,\mathbb{Z})$ be a triangular matrix, $A=(a_{ij})_{i,j=1}^{n}$.
    Then 
    \[
    |Q(\varphi_n(A))|\le \prod_{i\ne j}(1+|a_{ij}|).
    \]
\end{corollary}

In general, to obtain such an estimation a factorization is required. Moreover, such a factorization strongly depends on an algorithm applied. 
For instance, it is shown in~\cite{MR0737253} that for $n\ge 3$ each matrix $A \in SL(n,\mathbb{Z})$ is a product of at most $(3n^2 - n)/2 + 36$ elementary matrices.
However, elementary multipliers of the form $T_{ij}(k)$ may contain enormous $k$.

\section{Finite state representation of a free group}
\label{section_Automaton_representation_free_group}

Let $n=2$. We will show how isomorphic embedding $\varphi_2$ of $GL(2,\mathbb{Z})$ in $FAut \mathcal{T}_{4}$ give rise to an isomorphic embedding of the free group of rank 2 in $FAut \mathcal{T}_{2}$. Let $\mathsf{X}=\{0,1\}$ be the set of vertices of the first level of $\mathcal{T}_{2}$.  Consider finite state automorphisms $a,d \in FAut \mathcal{T}_{2}$ defined by their Moore diagrams,  see Figure~\ref{fig:free_group_generator1} and Figure~\ref{fig:free_group_generator2} correspondingly.

\begin{figure}[ht]
\centering
{\begin{tikzpicture}[>=stealth,scale=0.5,shorten >=2pt, thick,every initial by arrow/.style={*->}]
  \node[state, fill=gray!50] (a) at (0, 0)   {$a$};  
  \node[state] (a0) at (0, 5)   {$a_0$};
  \node[state] (a1) at (0, -5)   {$a_1$};
  \node[state] (b) at (9, 0)   {$b$};
  \node[state] (b0) at (9, 5)   {$b_0$};
  \node[state] (b1) at (9, -5)   {$b_1$};
  \node[state] (c) at (18, 0)   {$c$};
  \node[state] (c0) at (18, 5)   {$c_0$};
  \node[state] (c1) at (18, -5)   {$c_1$};
    \path[->,scale=2]
    (a) edge  [bend left] node [left]{$0|0$} (a0)
    (a0) edge  [bend left] node [right]{$0|0$} (a)
    (a) edge  [bend right] node [left]{$1|1$} (a1)
    (a1) edge  [bend right] node [right]{$1|1$} (a)
    (b) edge  [bend left] node [left]{$0|1$} (b0)
    (b0) edge  [bend left] node [right]{$0|1$} (b)
    (b) edge  [bend right] node [left]{$1|0$} (b1)
    (b1) edge  [bend right] node [right]{$0|1$} (b)
    (c) edge  [bend left] node [left]{$0|0$} (c0)
    (c0) edge  [bend left] node [right]{$1|1$} (c)
    (c) edge  [bend right] node [left]{$1|1$} (c1)
    (c1) edge  [bend right] node [right]{$0|0$} (c)
    (a0) edge   node [above]{$1|1$} (b)
    (a1) edge   node [above, pos=0.7]{$0|0$} (b)
    (b0) edge   node [below, pos=0.7]{$1|0$} (c)
    (b1) edge   node [above, pos=0.7]{$1|0$} (a)
    (c0) edge   node [below, pos=0.7]{$1|0$} (b)
    (c1) edge   node [below]{$1|1$} (b);    
\end{tikzpicture}}
\caption{Moore diagram $\mathcal{A}_1$ that defines generator $a$ of the free group}
\label{fig:free_group_generator1}
\end{figure}
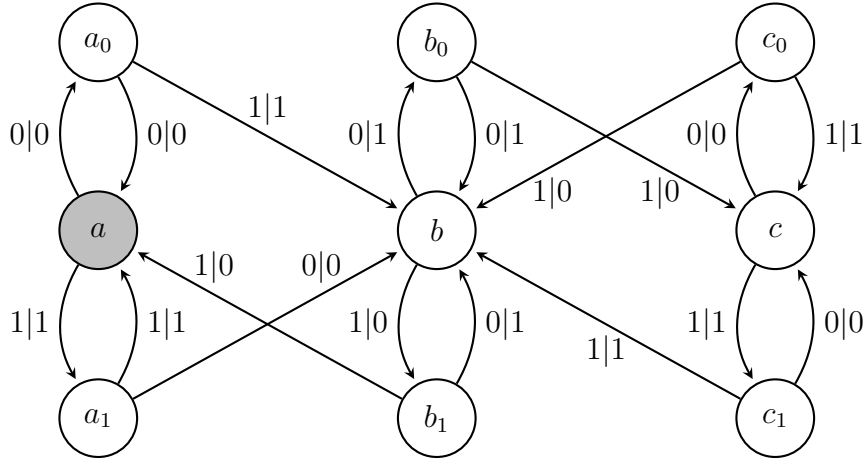

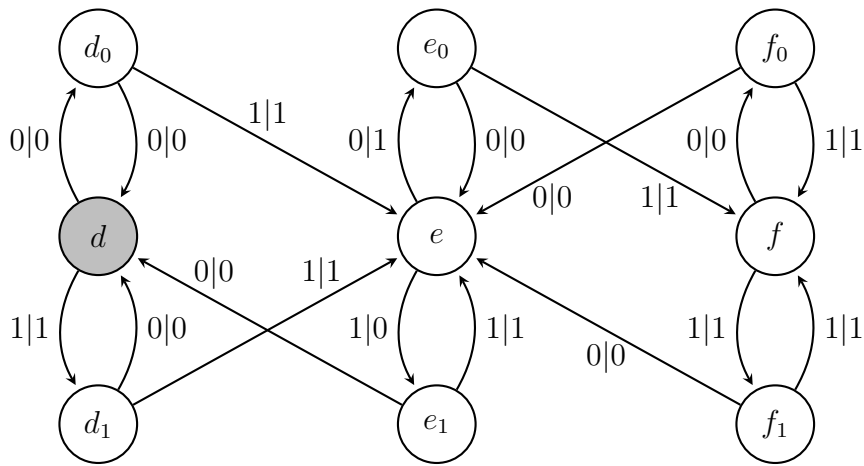
\begin{figure}[ht]
\centering
{\begin{tikzpicture}[>=stealth,scale=0.5,shorten >=2pt, thick,every initial by arrow/.style={*->}]
  \node[state, fill=gray!50] (d) at (0, 0)   {$d$};  
  \node[state] (d0) at (0, 5)   {$d_0$};
  \node[state] (d1) at (0, -5)   {$d_1$};
  \node[state] (e) at (9, 0)   {$e$};
  \node[state] (e0) at (9, 5)   {$e_0$};
  \node[state] (e1) at (9, -5)   {$e_1$};
  \node[state] (f) at (18, 0)   {$f$};
  \node[state] (f0) at (18, 5)   {$f_0$};
  \node[state] (f1) at (18, -5)   {$f_1$};
    \path[->,scale=2]
    (d) edge  [bend left] node [left]{$0|0$} (d0)
    (d0) edge  [bend left] node [right]{$0|0$} (d)
    (d) edge  [bend right] node [left]{$1|1$} (d1)
    (d1) edge  [bend right] node [right]{$0|0$} (d)
    (e) edge  [bend left] node [left]{$0|1$} (e0)
    (e0) edge  [bend left] node [right]{$0|0$} (e)
    (e) edge  [bend right] node [left]{$1|0$} (e1)
    (e1) edge  [bend right] node [right]{$1|1$} (e)
    (f) edge  [bend left] node [left]{$0|0$} (f0)
    (f0) edge  [bend left] node [right]{$1|1$} (f)
    (f) edge  [bend right] node [left]{$1|1$} (f1)
    (f1) edge  [bend right] node [right]{$1|1$} (f)
    (d0) edge   node [above]{$1|1$} (b)
    (d1) edge   node [above, pos=0.7]{$1|1$} (e)
    (e0) edge   node [below, pos=0.7]{$1|1$} (f)
    (e1) edge   node [above, pos=0.7]{$0|0$} (d)
    (f0) edge   node [below, pos=0.7]{$0|0$} (e)
    (f1) edge   node [below]{$0|0$} (b);    
\end{tikzpicture}}
\caption{Moore diagram $\mathcal{A}_2$ that defines generator $d$ of the free group}
\label{fig:free_group_generator2}
\end{figure}

\begin{theorem}
    \label{theorem_free_group}
    The group, generated by finite state automorphisms $a$ and $d$, if free of rank 2.
\end{theorem}

\begin{proof}
Let $\mathsf{X}=\{1,2,3,4\}$.
To simplify notation we use the following bijection from $\mathbb{Z}_2^2$ $\mathsf{X}$:
\[
(0,0) \mapsto 1, (1,0) \mapsto 2, (0,1) \mapsto 3,  (1,1) \mapsto 11.
\]
Then isomorphism $\varphi_2$ maps matrices
\[
\left(
\begin{matrix}
    1 & 0 \\
    1 & 1
\end{matrix}
\right), 
\left(
\begin{matrix}
    1 & 1 \\
    0 & 1
\end{matrix}
\right) \in GL(2,\mathbb{Z})
\]
to finite state automorphisms $t_1,s_1 \in FAut \mathcal{T}_{4}$ such that
\[
t_1 = (t_1,t_1,t_1,t_2)(34), \quad t_2 = (t_2,t_1,t_2,t_2)(12),
\]
\[
s_1 = (s_1,s_1,s_1,s_2)(24),  \quad
s_2 = (s_2,s_2,s_1,s_2)(13).
\]
Here $s_1=s_{12}\cdot t_1 \cdot s_{12}$.
Since matrices
\[
\left(
\begin{matrix}
    1 & 0 \\
    2 & 1
\end{matrix}
\right), 
\left(
\begin{matrix}
    1 & 2 \\
    0 & 1
\end{matrix}
\right)
\]
generate a free subgroup of rank 2 in  $GL(2,\mathbb{Z})$ (see \cite{MR0022557}) then their images under $\varphi_2$, finite state automorphisms 
$t_1^2$ and $s_1^2$, generate a free subgroup of rank 2 in $FAut \mathcal{T}_{4}$.
Direct computations show that
\[
t_1^2 = (t_1^2,t_1^2,t_1t_2,t_2t_1), \quad t_2^2 = (t_2t_1,t_1t_2,t_2^2,t_2^2),
\]
\[
t_1t_2=t_2t_1=(t_1t_2,t_1^2,t_1t_2,t_2^2)(12)(34),
\]
and
\[
s_1^2 = (s_1^2,s_1s_2,s_1^2,s_2s_1), \quad s_2^2 = (s_2s_1,s_2^2,s_1s_2,s_2^2),
\]
\[
s_1s_2=s_2s_1=(s_1s_2,s_1s_2,s_1^2,s_2^2)(13)(24).
\]

Consider the bijection between the set of vertices of the first level of $\mathcal{T}_{4}$ and the set of vertices of the second level of $\mathcal{T}_{2}$
defined by the rule
\[
1 \mapsto 00, 2 \mapsto 11, 3 \mapsto 10,  4 \mapsto 01.
\]
It defines an injection $f$ from the vertex set of $\mathcal{T}_{4}$ to $\mathcal{T}_{2}$. 
Then one directly verifies that for arbitrary vertex $v$ of $\mathcal{T}_{4}$
the following equalities holds
\[
f(v^{t_1^2})=(f(v))^{a}, \quad f(v^{t_1t_2})=(f(v))^{b}, \quad 
f(v^{t_2^2})=(f(v))^{c},
\]
\[
f(v^{s_1^2})=(f(v))^{d}, \quad f(v^{s_1s_2})=(f(v))^{d}, \quad 
f(v^{s_2^2})=(f(v))^{f}.
\]
It means that groups generated by the sets $\{t_1,s_1\}$ and $\{a,d\}$ are isomorphic as permutation groups. In particular, the latter group is free of rank 2. The proof is complete.
\end{proof}

\bibliographystyle{amsplain}
\bibliography{references.bib}



\end{document}